\documentclass[12pt]{amsart}
\usepackage[pagebackref=false,hyperindex=true,colorlinks=false, pdfauthor={Olivier Haution}, pdftitle={Reduced Steenrod operations and resolution of singularities}, pdfkeywords={  {S}teenrod operations, {R}iemann-{R}och theorem, {C}how groups}]{hyperref}

\usepackage{enumerate,amssymb,version}
\usepackage{geometry}
\usepackage[all]{xy}

\newcommand{\adp}{\psi_p}
\newcommand{\ad}{\psi^p}
\newcommand{\Th}{\theta^p}

\newcommand{\Z}{\mathbb{Z}}
\newcommand{\Zu}{\mathbb{Z}[1/p]}
\DeclareMathOperator{\Aut}{Aut}
\DeclareMathOperator{\Hom}{Hom}

\newcommand{\Lc}{\mathcal{L}}
\newcommand{\pp}{p}
\newcommand{\K}{K}
\newcommand{\id}{\mathrm{id}}
\newcommand{\pt}{\mathsf{point}}

\DeclareMathOperator{\Fi}{F}
\newcommand{\Fig}[2]{\Fi_{\gamma}^{#1} \K^0({#2})}
\newcommand{\Fit}[2]{\Kzl({#2})_{({#1})}}

\newcommand{\Grg}[2]{\gr_{\gamma}^{#1} \K^0({#2})}
\newcommand{\Grgs}{\gr_{\gamma} \K^0}
\newcommand{\Grgd}[1]{\gr_{\gamma}^{#1} \K^0}

\newcommand{\wg}[2]{w_{#1}^{\gamma,\pp}({#2})}
\newcommand{\wgs}{w^{\gamma,\pp}}

\newcommand{\wchs}{w^{\CH,\pp}}
\newcommand{\wch}[2]{ w_{#1}^{\CH,\pp}({#2})}

\DeclareMathOperator{\rank}{rank}
\DeclareMathOperator{\gr}{gr}
\DeclareMathOperator{\Spec}{Spec}
\DeclareMathOperator{\im}{im}

\newcommand{\Q}{\mathbb{Q}}

\newcommand{\ba}[1]{\overline{#1}}

\newcommand{\Kzh}{\K^0}
\newcommand{\Kzl}{\K_0}
\newcommand{\Kit}[2]{\Kzl({#2})_{({#1})}}
\newcommand{\Grt}[2]{\gr_{{#1}}\Kzl({#2})}
\newcommand{\Grts}{\gr_{\bullet}\Kzl}

\DeclareMathOperator{\CH}{CH}
\DeclareMathOperator{\Ch}{Ch}
\newcommand{\Chb}{\ba{\Ch}}

\newcommand{\Fc}{\mathcal{F}}
\newcommand{\Ec}{\mathcal{E}}
\newcommand{\Oc}{\mathcal{O}}

\newcommand{\Tan}{T}

\DeclareMathOperator{\Steenrod}{S}
\newcommand{\Sq}{\ba{\Steenrod}}

\newtheorem{theorem}{Theorem}[section]
\newtheorem{proposition}[theorem]{Proposition}
\newtheorem{lemma}[theorem]{Lemma}

\theoremstyle{definition}
\newtheorem{remark}[theorem]{Remark}
\newtheorem{example}[theorem]{Example}
\newtheorem{definition}[theorem]{Definition}
\newtheorem{condition}[theorem]{Condition}

\begin{document}
\begin{abstract}
We give a new construction of a weak form of Steenrod operations for Chow groups modulo a prime number $p$ for a certain class of varieties. This class contains projective homogeneous varieties which are either split or over a field admitting some form of resolution of singularities, for example any field of characteristic not $p$. These reduced Steenrod operations are sufficient for some applications to the theory of quadratic forms. 
\end{abstract}
\author{Olivier Haution}
\title{Reduced Steenrod operations and resolution of singularities}
\email{olivier.haution@gmail.com}
\address{School of Mathematical Sciences\\
University of Nottingham\\
University Park\\
Nottingham\\
NG7 2RD\\ 
United Kingdom}
\date{\today}
\maketitle

\tableofcontents
\section*{Introduction}
The Steenrod operations have first been introduced in algebraic topology, in the late forties. These are operations on the modulo $p$ cohomology of topological spaces. Steenrod, Serre, Borel and others used them to prove results in algebraic topology, especially obstructions to the existence of bundles and maps. These operations have been later used in a different way, in the context of the Adams spectral sequence or of the Sullivan conjecture.

Fifty years after Steenrod original construction, the operations appeared in algebraic geometry, and were quickly used to study projective homogeneous varieties, or by Voevodsky to prove the Milnor (and Bloch-Kato) conjecture.

Unfortunately, there is no construction of Steenrod operations modulo $p$ which works over a field of characteristic $p$. As a result, many important questions regarding projective homogeneous varieties remain open for certain values of the characteristic of the base field. For example, some of deepest theorems on quadratic forms are unknown when the base field has characteristic two.\\

Constructions of Steenrod operations for Chow groups modulo a prime number $p$ are described in the papers \cite{Boi-A-08, Bro-St-03, EKM, Vo-03}. They all involve considering the action of the cyclic group of order $p$ on the product of $p$ copies of a given scheme, as did the original construction of Steenrod. This kind of approach fails over a field of characteristic $p$.

Over a field of characteristic zero, Steenrod operations may be constructed using techniques related to algebraic cobordism, as in \cite{Lev-St-05}. As explained in \cite[Remark~1.2.20]{LM-Al-07} or \cite[\S~2.7]{LP-Alg-09}, both resolution of singularities and the weak factorization theorem of \cite{Weak-factorization} are needed to compute the cobordism ring of a point, and to prove the expected universality property of the Chow group among theories whose formal group law is additive. This accounts for the restriction to base fields of characteristic zero.\\ 

When $X$ is a projective homogeneous variety over a field $F$ (under a semi-simple algebraic group), and $\ba{F}/F$ a field extension splitting $X$, the subgroup of the Chow group modulo $p$ of $X \times_{\Spec(F)} \Spec(\ba{F})$ which consists of the cycles defined over $F$ is the \emph{reduced Chow group} of $X$. A weak version of Steenrod operations, when they exist, are the \emph{reduced Steenrod operations}, that is, the operations induced at the level of reduced Chow groups.

Steenrod operations have been used as mean of producing new algebraic cycles, especially on the square of a given projective homogeneous variety, thus providing a motivic decomposition of this variety. Thanks to the Rost nilpotence theorem, the same result can be achieved using only reduced Steenrod operations. For example, it is explained in \cite{firstsq} that reduced Steenrod operations are sufficient to prove statements such as Hoffmann's Conjecture on the possible values of the first Witt index of quadratic forms. However some subtler results do not seem to follow from the existence of reduced Steenrod operations, for example the classical degree formula and its consequences, or \cite[Theorem 9.1]{Ros-On-06} about Rost correspondences.\\

One of the purposes of this paper is to show that, in order to construct reduced Steenrod operations, one does not need the weak factorization theorem, nor any explicit restriction on the characteristic of the base field. This contrasts with all known constructions of these operations. We still need some kind of desingularization procedure (a pro-$p$ version of resolution of singularities suffices).

Given a split projective homogeneous variety over an arbitrary field, we build homological Steenrod operations on the Chow group (modulo a prime integer $\pp$) of the variety, satisfying the expected functoriality properties. Of course, one wants rather to study non-split projective homogeneous varieties, and one cannot expect to derive deep results from this construction. However this reduces the question of constructing reduced Steenrod operations to the question of showing that a certain well-defined operation on the Chow group of a projective homogeneous variety, with scalars extended to a splitting field, preserves rationality of cycles. This suggests that constructing the reduced form of Steenrod operations might be easier in general.

If the base field admits $\pp$-resolution of singularities (Definition~\ref{def:p-res}), we show that the homological Steenrod operations that we have constructed over a splitting field send rational cycles to rational cycles. Therefore we get the reduced Steenrod operations.\\

In 1966, Atiyah published a paper, \cite{Ati-Po-66}, where Adams operations are reconstructed using equivariant $K$-theory, following a pattern very similar to Steenrod construction of its operations. When a finite CW-complex has no torsion in its singular cohomology, then the latter coincides with the graded group associated with the skeletal filtration in $K$-theory. Atiyah explains how to recover Steenrod operations in cohomology from Adams operations in $K$-theory, under this (very restrictive) assumption of absence of torsion. 

The basic idea of the present paper is to translate some of Atiyah's constructions for topological $K$-theory of finite CW-complexes to the setting of algebraic geometry. Adams operations are defined on the (cohomological) $K$-theory of vector bundles, and the topological filtration on the (homological) $K$-theory of coherent sheaves. But the two theories coincide on regular schemes, and interesting phenomena can be expected to happen from the interaction of these two notions. A precise statement is given in Theorem~\ref{ad_main}.\\

We believe that Section~\ref{Sect:charac} can provide good preview of the techniques used here. Some divisibilities of characteristic numbers (for example the fact that the Segre number of a smooth projective variety is divisible by two) are usually established using Steenrod operations. We show in this section how these divisibilities follow from the Adams-Riemann-Roch theorem, therefore proving them over fields of arbitrary characteristic.\\

In the last section, we explain how the results obtained here can be used to prove a degree formula. This generalizes a formula of \cite{Zai-09}, when the base field admits $p$-resolution of singularities. It is quite surprising that we can recover this formula, since \cite{Zai-09} uses in an essential way the \emph{non-reduced} Steenrod operations, via Rost's degree formula. \\

\paragraph{\bf Acknowledgements} This work is part of my Ph.D.~thesis, defended in December 2009 at the university of Paris 6, under the direction of Nikita Karpenko. I would like to thank him for his guidance during this work.

The idea, crucial in this work, of transposing Atiyah's results as explained above, has been suggested by Fabien Morel.

I am grateful to the referee for his comments, especially those concerning the historical background of the subject of this paper, which have been incorporated in the introduction.

\section{Notations}
The letter $p$ will denote a prime integer, unless otherwise specified. When $a$ is a rational number, the integer $[a]$ is defined as the greatest integer $\leq a$.

If $F$ is a functor from a category $\mathsf{C}$ to the category of commutative rings, we denote by $F^\times$ the induced functor with values in the category of abelian groups, where for an object $C \in \mathsf{C}$, the group $F^\times(C)$ is the multiplicative group of the ring $F(C)$.\\

\paragraph{\bf Varieties} A \emph{variety} is a separated quasi-projective scheme of finite type over a field, which need not be irreducible nor reduced. All schemes considered here will be assumed to be varieties. A \emph{regular variety} is a variety whose local rings are all regular local rings.

A \emph{local complete intersection morphism} $f \colon X \to Y$ is a morphism which can be factored as a regular closed embedding $c$, followed by a smooth morphism $s$. If $c$ has normal bundle $N_c$, and $s$ has tangent bundle $\Tan_s$, then the \emph{virtual tangent bundle} of $f$ is defined as $\Tan_f=c^*[\Tan_s] - [N_c] \in \Kzh(X)$.

By $\pt$ we shall denote the spectrum of the base field. We say that a variety $X$ is a \emph{local complete intersection} when its structural morphism $x \colon X \to \pt$ is a local complete intersection morphism. Its virtual tangent bundle $\Tan_X$ is the virtual tangent bundle of $x$.\\

\paragraph{\bf Grothendieck group of coherent sheaves} Let $X$ be a variety. We shall write $\Kzl(X)$ for the Grothendieck group of coherent $\Oc_X$-sheaves. A sheaf $\Fc$ has a class $[\Fc] \in \Kzl(X)$. Projective morphisms induce push-forwards, and flat morphisms or regular closed embeddings  induce pull-backs.\\

\paragraph{\bf Topological filtration}
The group $\Kzl(X)$ is endowed with a topological filtration $0=\Kit{-1}{X} \subset \Kit{0}{X} \subset \cdots \subset \Kit{\dim X}{X} = \Kzl(X)$. For every integer $d$, $\Kit{d}{X}$ is the subgroup of $\Kzl(X)$ generated by classes of coherent $\Oc_X$-sheaves supported in dimension $\leq d$. This filtration is compatible with projective push-forwards and pull-backs along flat morphisms or regular closed embeddings (\cite[Theorem~83]{gil-K-05}). The associated graded group shall be denoted by $\Grt{\bullet}{X}$.\\

\paragraph{\bf Grothendieck ring of vector bundles}
For a variety $X$, we shall write $\Kzh(X)$ for its Grothendieck ring of vector bundles. A vector bundle $\Ec$ on $X$ has a class $[\Ec] \in \Kzh(X)$. An arbitrary morphism of varieties $f \colon X \to Y$ induces a pull-back homomorphism $f^* \colon \Kzh(Y) \to \Kzh(X)$.

The ring $\Kzh(X)$ acts on $\Kzl(X)$, and the map $\Kzh(X) \to \Kzl(X), x \mapsto x \cdot [\Oc_X]$ is an isomorphism when the variety $X$ is regular (\cite[Chapter~VI, Proposition~3.1]{FL-Ri-85}). We shall therefore sometimes identify the two Grothendieck groups for regular varieties.

When $X$ is connected, $\Kzh(X)$ is equipped with an augmentation 
\[
\rank \colon \Kzh(X) \to \Z.
\]\\

\paragraph{\bf Gamma filtration}
The group $\Kzh(X)$ is endowed with the gamma filtration $\cdots \subset \Fig{k+1}{X} \subset \Fig{k}{X} \subset \cdots \subset \Fig{0}{X}=\Kzh(X)$.  We have (\cite[Chapter~V, Corollary~3.10]{FL-Ri-85})
\begin{equation}
\label{eq:fig}
\Fig{\dim X +1}{X}=0.
\end{equation}
The gamma filtration is compatible with pull-backs, products, and we have
\begin{equation}
\label{eq:gammatop}
\Fig{k}{X} \cdot \Fit{n}{X}  \subset \Fit{n-k}{X}.
\end{equation}
We shall write $\Grg{\bullet}{X}$ for the associated graded ring. It acts on the graded group $\Grt{\bullet}{X}$.\\

\paragraph{\bf Chern classes}
The presheaves $\Kzh$ and $\Grgs$ are equipped with Chern classes (see \cite[\S~2]{Kar-Co-98} where the category of smooth varieties is considered; the same construction works for the category of arbitrary varieties). We shall use the following formula for the first Chern class of a line bundle $\Lc$ on a variety $X$ (the sheaf $\Lc^\vee$ being the dual of $\Lc$) : 
\[
c_1[\Lc]=c_1(\Lc)=1-[\Lc^\vee] \in \Kzh(X).
\]
This element belongs to $\Fig{1}{X}$, hence is nilpotent by \eqref{eq:fig}.

\section{Bott's Class}
Let $\Lc$ be a line bundle on a variety $X$. The element
\[
1+[\Lc]^{-1}+\cdots+[\Lc]^{1-p}=p-c_1(\Lc^{\otimes 1})-\cdots-c_1(\Lc^{\otimes (p-1)})
\]
is invertible as an element of $\Zu \otimes \Kzh(X)$. Using the Chern classes and the splitting principle, we see that there is a unique morphism of presheaves of groups
\[
\Th \colon \Kzh \to (\Zu \otimes \Kzh)^{\times}
\]
satisfying, for any line bundle $\Lc$, the formula
\[
\Th[\Lc]=1+[\Lc]^{-1}+\cdots+[\Lc]^{1-p}=\frac{1-\big(1-c_1(\Lc)\big)^p}{c_1(\Lc)}.
\]
In the terminology of \cite{Pan-Ri-02}, $\Th$ is the \emph{inverse Todd genus} of the $p$-th Adams operation.

We also consider the unique morphism of presheaves of groups
\[
\wgs \colon \Kzh \to (\Grgs)^{\times}
\]
satisfying, for any line bundle $\Lc$, the formula
\begin{equation}
  \label{eq:wg}
  \wgs[\Lc]=\big(-c_1(\Lc)\big)^{p-1}+1.
\end{equation}
Note that this is --- up to a sign --- the total Chern class homomorphism when $p=2$. The grading on $\Grgs$ allows us to define individual components 
\[
\wgs_k \colon \Kzh \to \Grgd{k(p-1)}.
\]

Finally, using the Chern classes $c_i$ with values in Chow groups, we can consider the unique morphism of presheaves of groups
\[
\wchs \colon \Kzh \to \Aut(\CH)
\]
satisfying $\wchs[\Lc]=\big(-c_1(\Lc)\big)^{p-1} + \id$, and individual components $\wchs_k$ with values in $\Hom(\CH_{\bullet},\CH_{\bullet-k(p-1)})$.\\ 

If $x \in \Kzh(X)$, we shall sometimes denote the same way the endomorphism $\wchs(x)$ of $\CH(X)$ and the element $\wchs(x)[X] \in \CH(X)$, and similarly for Chern classes.

\begin{proposition}
\label{prop:Bott}
Let $X$ be a connected variety, and $e \in \Kzh(X)$. Then there exists elements $e_k \in \Fig{k(\pp-1)}{X}$  satisfying in $\Zu \otimes \Kzh(X)$
\[
\Th(e)=\sum_{k=0}^{\left[\frac{\dim X}{p-1}\right]} \pp^{\rank(e) -k} \otimes e_k.
\]
Moreover, the element $e_k$ can be chosen so that we have in $\Grg{k(\pp-1)}{X}$,
\[
e_k=\wg{k}{e} \mod \Fig{k(\pp-1)+1}{X}.
\]
\end{proposition}
\begin{proof}
It is easy to see that there exists a polynomial $T_{\pp}$, with integral coefficients, such that we have, for any line bundle $\Lc$,\footnote{This is precisely the place where we need $\pp$ to be prime number, and not an arbitrary integer}
\[
\Th[\Lc]=\big(-c_1(\Lc)\big)^{\pp-1} + \pp \Big(1 + c_1(\Lc) \cdot T_{\pp}\big(c_1(\Lc)\big)\Big).
\]
The statement can then be deduced from the comparison of this formula with \eqref{eq:wg}, using the splitting principle (and \eqref{eq:fig}).
\end{proof}

\section{Topological filtration and Chow group}
For every variety $X$, there is a homomorphism of graded groups 
\[
\varphi_X \colon \CH_{\bullet}(X) \to \Grt{\bullet}{X},
\]
sending the class of a $k$-dimensional closed subvariety $Z$ of $X$ to the class of its structure sheaf $\Oc_Z$ modulo $\Fit{k-1}{X}$ (see \cite[Example~15.1.5]{Ful-In-98}). We record here some properties of this map that will be useful in the sequel.

\begin{proposition}
\label{prop:phi}
\begin{enumerate}[(i)]
\item \label{phiproper} If $f \colon X \to Y$ is proper, then $f_* \circ \varphi_X=\varphi_Y \circ f_*$.
\item \label{phiflat} If $f \colon X \to Y$ is a flat morphism, then $f^* \circ \varphi_Y=\varphi_X \circ f^*$.
\item \label{philci} If $f \colon X \to Y$ is a local complete intersection morphism, then $f^* \circ \varphi_Y=\varphi_X \circ f^*$.
\item \label{phiextprod} If $X$ and $Y$ are varieties over a common base field, $x\in \CH(X), y \in \CH(Y)$, we have $\varphi_{X\times Y}(x \boxtimes y)=\varphi_X(x) \times \varphi_Y(y)$.
\item \label{phiw} We have a commutative diagram
\[ \xymatrix{
  \Kzh \ar[r]^{\wchs} \ar[d]_{\wgs} & \CH \ar[d]^{\varphi} \\ 
\Grgs \ar[r] & \Grts
}\]
\item \label{phiisom} The map $\varphi_X$ is surjective, its kernel consists of torsion elements.
\end{enumerate}
\end{proposition}
\begin{proof}
For \eqref{phiproper} and surjectivity in \eqref{phiisom}, see \cite[Example~15.1.5]{Ful-In-98}. Properties \eqref{phiextprod} and \eqref{phiflat} are easy to prove, and \eqref{phiw} is similar to \cite[Lemma~2.16]{Kar-Co-98}.\\

We now prove \eqref{philci}. Let $X \hookrightarrow Y$ be a regular closed embedding with normal bundle $N$. The map $\sigma^{\Kzl} \colon \Kzl(Y)\to \Kzl(N)$, called specialization in \cite[p.283]{gil-K-05} or \cite[p.352]{Ful-In-98}, respects the topological filtration, hence induces a deformation homomorphism $\sigma^{\gr \Kzl} \colon \Grts(Y) \to \Grts(N)$. Similarly, there is a deformation homomorphism $\sigma^{\CH} \colon \CH(Y) \to \CH(N)$ (see \cite[\S 5.2]{Ful-In-98}, or \cite[\S 51]{EKM}). Let $Z \hookrightarrow Y$ be a closed embedding. The normal cone $C$ of the induced closed embedding $Z\cap X \hookrightarrow Z$ can be viewed as a closed subvariety of $N$. We claim that
\[
\sigma^{\CH}[Z]=[C] \quad \text{ and } \quad \sigma^{\Kzl}[\Oc_Z]=[\Oc_C].
\]
This is immediate if one uses the definitions of \cite{Ful-In-98} for the deformation homomorphisms. Otherwise, the first formula is \cite[Proposition~52.7]{EKM}, while the second one can be obtained with the same proof. It follows that
\[
\varphi_N \circ \sigma^{\CH}[Z]=\varphi_N [C]=[\Oc_C]=\sigma^{\gr \Kzl}[\Oc_Z]=\sigma^{\gr \Kzl} \circ \varphi_Y[Z].
\]

Since pull-backs along local complete intersection morphisms are constructed out of deformation homomorphisms and flat pull-backs in both theories (\cite[Proposition~86]{gil-K-05} and \cite[\S 55.A]{EKM}) we see that $\varphi$ is compatible with pull-backs along local complete intersection morphisms. This proves \eqref{philci}.\\

It remains to prove the second part of \eqref{phiisom}, which is done in \cite[Example~15.3.6]{Ful-In-98} under the additional hypothesis that $X$ is smooth. When $X$ is arbitrary, we will provide a left inverse of $\varphi_X \otimes \Q$.

In \cite[Theorem~18.3]{Ful-In-98}, a natural transformation
\[
\tau \colon \Kzl \to \CH \otimes \Q
\]
of functors from the category of varieties and proper morphisms to the category of abelian groups is constructed. Consider both on $\Kzl(X)$ and on $\CH(X)\otimes \Q$ the filtration whose $k$-term is generated by $i_*\Kzl(Z)$ and $i_*\CH(Z) \otimes \Q$ for all closed embeddings $i \colon Z \hookrightarrow X$ with $\dim Z \leq k$. The associated graded groups are $\Grt{\bullet}{X}$ and $\CH_{\bullet}(X) \otimes \Q$, and since $\tau$ commutes with push-forwards, it induces a morphism of graded abelian groups
\[
\gr(\tau_X)_\Q \colon \Grt{\bullet}{X}\otimes \Q \to \CH_{\bullet}(X) \otimes \Q.
\]
It is also proven in \cite[Theorem~18.3]{Ful-In-98} that, for a closed $k$-dimensional subvariety $Z$ of $X$, we have 
\[
\tau_X[\Oc_Z]=[Z] \mod \CH_{<k}(X)\otimes \Q,
\]
hence $\gr(\tau_X)_\Q \circ (\varphi_X \otimes \Q)=\id_{\CH(X) \otimes \Q}$, as required.
\end{proof}

\section{Divisibility of some characteristic numbers}
\label{Sect:charac}
In this section we prove a result of independent interest about some characteristic numbers of varieties over an arbitrary field. This may provide some motivation for the sequel, see Remark~\ref{rem:segre} below.\\ 

When $X$ is a projective variety we consider the degree morphism 
\begin{equation}
  \label{eq:deg}
\deg \colon \Kzl(X) \to \Kzl(\pt) \xrightarrow{\rank} \Z,
\end{equation}
where the first map is the push-forward along the structural morphism. Since the map $\varphi \colon \CH_\bullet \to \Grts$ is compatible with proper push-forwards (Proposition~\ref{prop:phi},~\eqref{phiproper}), and respects the isomorphisms $\Kzl(\pt)\to \Z$ and $\CH(\pt) \to \Z$, we have a commutative diagram
\begin{equation}
  \label{eq:degcomp}
\xymatrix{
\CH_0(X)\ar[rr]^{\varphi_X} \ar[dr]_{\deg} && \Grt{0}{X}=\Fit{0}{X} \ar[dl]^{\deg} \\ 
 &\Z& 
}
\end{equation}

\begin{proposition}
\label{prop:segre}
Let $p$ be a prime number.  Let $X$ be a connected projective local complete intersection variety of positive dimension $k(p-1)$, with virtual tangent bundle $\Tan_X \in \Kzh(X)$. Then the degree of the class $\wch{k}{-\Tan_X}$ is divisible by $p$.
\end{proposition}
\begin{proof} 
We apply the Adams Riemann-Roch theorem with denominators (\cite[Chapter~V, Theorem~7.6]{FL-Ri-85}) to the structural morphism $X \to \pt$. This gives a commutative diagram
\[
\xymatrix{
\Zu \otimes \Kzh(X) \ar[d]^{\deg} \ar[rrr]^{\Th(-\Tan_X) \circ (\id \otimes \ad)} &&& \Zu \otimes \Kzh(X) \ar[d]^{\deg}\\
\Zu \ar[rrr]_{\ad=\id} &&& \Zu.
}
\]
In particular
\begin{equation}
\label{degp}
\deg \circ \Th(-\Tan_X)=\deg[\Oc_X] \in \Z \subset \Zu.
\end{equation}

Let $\widetilde{\K}(X)$ be the quotient of $\Kzh(X)$ by its $p$-power torsion subgroup. Note that the degree map factors through $\widetilde{\K}(X)$. We view $\widetilde{\K}(X)$ as a subgroup of $\Zu \otimes \Kzh(X)$. Then by Proposition~\ref{prop:Bott}, the element $\pp^{kp} \cdot \Th(-\Tan_X)$ belongs to $\widetilde{\K}(X)$, and is congruent modulo $\pp \cdot \widetilde{\K}(X)$ to 
\[
\wg{k}{-\Tan_X} \in \Grg{k(p-1)}{X}=\Fig{k(p-1)}{X}.
\]
By \eqref{degp}, since $k\geq 1$, its degree has to be divisible by $\pp$. Using Proposition~\ref{prop:phi},~\eqref{phiw} and \eqref{eq:degcomp}, we see that $\deg \wg{k}{-\Tan_X}$ coincides with the degree of $\wch{k}{-\Tan_X}$. This proves the claim. 
\end{proof}

\begin{remark}
\label{rem:segre}
Under some assumptions about the characteristic of the base field, this lemma is well-known. More general forms have been proven using other methods : Steenrod operations in \cite{Mer-St-03} and \cite[\S~2.6.7]{Pan-Ri-02}, or algebraic cobordism in \cite{Lev-St-05}. 

When the characteristic of the base field is $p$, the case $p=2$ has been considered in \cite{Ros-On-08}, using the Frobenius map.
\end{remark}

\section{Torsion-free varieties, resolution of singularities}
\paragraph{\bf Torsion-free varieties}
We say that a variety $X$ is \emph{torsion-free} if its Chow group is without torsion. 
\begin{example}
A variety whose Chow motive splits as a sum of twisted Tate motives is torsion-free. In particular a cellular variety is torsion-free.
\end{example}

\begin{remark}
\label{rem:torfree}
It follows from Proposition~\ref{prop:phi}, \eqref{phiisom}, that when $X$ is torsion-free, the group $\Grt{\bullet}{X}$ is also torsion-free, being isomorphic to $\CH_{\bullet}(X)$.
\end{remark}

We will only apply Lemma~\ref{injl} and Proposition~\ref{cor:injl} for $p$ a prime integer, but this assumption is not required for these two statements.  

\begin{lemma}
\label{injl}
Let $\cdots \supset K_{n+1} \supset K_n \supset \cdots$ be a filtered group such that $\cap_i K_i=0$. Assume that the associated graded group $\gr_{\bullet}K$ is without torsion.  Set $K=\cup_i K_i$, and let $x \in K_n$ and $y \in K$ be elements satisfying, for some non-zero integers $p$ and $m$, the congruence
\[
m \cdot x=m p \cdot y \mod K_{n-1}.
\]
Then the image of $x$ in $\Z/p \otimes \gr_{n}K$ is zero.
\end{lemma}
\begin{proof}
Since $\gr_{\bullet}K$ is a torsion-free group, it follows that if $y=0$, then $x \in K_{n-1}$, and its image in the $n$-th graded group is zero.

Hence we may assume $y \neq 0$. Let $d$ be the integer such that $y \in K_{d}$ and $y \notin K_{d-1}$ (such an integer certainly exists because $y\notin \cap_i \K_{i}=0$). Assume $d > n$. Then $m p \cdot y \in K_{n} \subset K_{d-1}$, hence $m p \cdot y$ is zero in the torsion-free group $\gr_{d}K$, which is incompatible with the choice of the integer $d$. Hence $d \leq n$ and  $y \in K_{n}$. Using the fact that $\gr_{n}K$ has no torsion, the assertion follows.
\end{proof}

\begin{proposition}
\label{cor:injl}
Let $X$ be a torsion-free variety, and $p$ an integer. Let $u_k \in \Fit{d-k(p-1)}{X}$ for $k=0,\cdots,d$ and assume that there is an element $u_{-1} \in \Kzl(X)$ such that
\[
\sum_{k=-1}^{d} p^{-k} \otimes u_k=0 \in \Zu \otimes \Kzl(X).
\]
Then for all $k=0,\cdots,d$, the image of $u_k$ in $\Z/p \otimes \Grt{d-k(p-1)}{X}$ is zero. 
\end{proposition}
\begin{proof}
In view of Remark~\ref{rem:torfree}, the group $\Grts(X)$ is torsion-free. Therefore the group $\Kzl(X)$ has to be also torsion-free. We have for all $k=0,\cdots,d$ a congruence,
\[
p^{d-k} \cdot u_k = p^{d-k} p \cdot \left(- \sum_{i=-1}^{k-1} p^{k-i-1} \cdot u_i\right) \mod \Fit{d-k(p-1)-1}{X}.
\]
We are in position to apply Lemma~\ref{injl} above, with $K_\bullet=\Fit{\bullet}{X}$, $n=d-k(p-1)$, $m=p^{d-k}$ and $x=u_k$.
\end{proof}

\paragraph{\bf Resolution of singularities} An \emph{alteration} of an integral variety $X$ is an integral variety $X'$ with a projective generically finite morphism $X' \to X$, and $\dim X'=\dim X$. The \emph{degree} of the alteration $X'$ is the degree of the extension of function fields $F(X')/F(X)$.

\begin{definition}
\label{def:p-res}
Given a prime integer $\pp$, we shall say that a field $F$ \emph{admits $\pp$-resolution of singularities} if every integral variety $X$ over $F$ admits a regular alteration of degree prime to $\pp$.
\end{definition}

\begin{remark}
Ofer Gabber has proved that any field of characteristic not $\pp$ admits $\pp$-resolution of singularities (see \cite{Gabber}). We shall not explicitly use this result here.
\end{remark}

We shall say that a variety $X$ is \emph{generated by $p$-regular classes}, when the group $\Grt{\bullet}{X}$ is generated by elements $z$, satisfying $\lambda \cdot  z=f_*[\Oc_Z]$, for some (depending on $z$) integer $\lambda$ prime to $p$, and some projective morphism $f \colon Z \to X$, with $Z$ a local complete intersection variety.\\

\begin{example}
\label{ex:vb}
Assume that a variety $X$ is generated by $p$-regular classes. If $\pi \colon E \to X$ is a vector bundle, then $E$ is also generated by $p$-regular classes. This is a consequence of the following statements:

---  Assuming that $\pi$ is of constant rank $e$, the pull-back $\pi^* \colon \Grt{\bullet}{X} \to \Grt{\bullet+e}{E}$ is an isomorphism (\cite[Lemma~85]{gil-K-05}).

--- Let $Z \to X$ be a projective morphism. If $Z$ is a local complete intersection variety, then so is $Z \times_X E$.
\end{example}

We consider the following condition on a variety $X$ over a field $F$ and a prime integer $\pp$:
\begin{condition}
\label{def:cond}
The variety $X$ is generated by $\pp$-regular classes, and there is a field extension $\ba{F}$ of $F$ such that the variety $\ba{X}:=X \times_{\Spec(F)} \Spec(\ba{F})$ is torsion-free. We require additionally that $\ba{X}$ be generated by $\pp$-regular classes.
\end{condition}

We now give two examples of varieties satisfying Condition~\ref{def:cond}.

\begin{example}
\label{ex:split_phv}
Let $X$ be a split projective homogeneous variety under a semi-simple algebraic group over an arbitrary field $F$. Set $\ba{F}=F$. Then the Chow motive of the variety $X=\ba{X}$ splits as a sum of Tate motives by the results of \cite{Ko-91}, hence the variety $X$ is torsion-free. The ring $\CH_\bullet(X)$ and $\Grt{\bullet}{X}$ are then isomorphic, and using the results of \cite{Dem-De-74}, we see that $X$ is generated by $\pp$-regular classes, for any prime integer $\pp$. 
\end{example}

\begin{example}
\label{ex:phv}
Let $X$ be a projective homogeneous variety under a semi-simple algebraic group over a field $F$ admitting $\pp$-resolution of singularities. Let $\ba{F}$ be an algebraic closure of $F$. Then $X$ is generated by $\pp$-regular classes. By Example~\ref{ex:split_phv}, the split projective homogeneous variety $\ba{X}$ is torsion-free, and $\ba{X}$ is generated by $\pp$-regular classes. 
\end{example}

\section{Adams operations and the topological filtration}
When $f \colon X \to Y$ is a proper (\emph{resp.} local complete intersection) morphism, we shall also write $f_*$ (\emph{resp.} $f^*$) for the map $\id_{\Zu} \otimes f_*$ (\emph{resp.} $\id_{\Zu} \otimes f^*$).\\

\paragraph{\bf Homological Adams operations}(See  \cite[Th\'eor\`eme~7]{Sou-Op-85})  These are operations 
\[
\adp \colon \Kzl \to \Zu \otimes \Kzl
\]
commuting with proper push-forwards, external products, and satisfying the formula, for $f\colon X \to Y$ a local complete intersection morphism with virtual tangent bundle $\Tan_f \in \Kzh(X)$,
\begin{equation}
  \label{eq:RR}
  \adp \circ f^*=\Th(-\Tan_f) \circ f^* \circ \adp.
\end{equation}
In addition, $\adp\colon \Kzl(\pt)\to \Zu\otimes\Kzl(\pt)$ is the inclusion $\Z \subset \Zu$.

These operations are constructed as follows. Let $X$ be variety. We choose an embedding $j \colon X \hookrightarrow M$ in a regular variety $M$. Then we write $\Kzh_X(M)$ for the Grothendieck group of bounded complexes of locally free $\Oc_M$-modules acyclic off $X$, modulo the subgroup generated by acyclic complexes. One can then construct a $p$-th Adams operation ``with supports'' $\ad \colon \Kzh_X(M) \to \Kzh_X(M)$. Note that while the construction of this operation is not completely immediate in general (see for example \cite[Chapters 3 and 4]{GS-87}), the situation becomes considerably simpler over fields of characteristic $p$, thanks to the Frobenius morphism.

There is a cap product $\cap \colon \Kzh_X(M) \otimes \Kzl(M) \to \Kzl(X)$, which induces an isomorphism $-\cap [\Oc_M] \colon \Kzh_X(M) \to \Kzl(X)$ because $M$ is regular (\cite[Lemma~1.9]{GS-87}).

Now let $x \in \Kzl(X)$, and $y \in \Kzh_X(M)$ the element such that $y \cap[\Oc_M]=x$. Let $\Tan_M \in \Kzh(M)$ be the virtual tangent bundle of $M$, and set
\begin{equation}
\label{eq:homad}
\adp(x)=\Th(-j^*\Tan_M) \cdot \big( \ad(y) \cap [\Oc_M]\big).
\end{equation}
Using the Riemann-Roch theorem for the Adams operation with supports, one can check that the value of $\adp(x)$ does not depend on the choices of $j$ and $M$. See \cite[Proposition~2.8]{duality} for a detailed proof of the formula \eqref{eq:RR}.\\

We now record a property of the homological Adams operation.
\begin{lemma}
\label{lemm:integrality}
Let $X$ be a variety and $x \in \Fit{d}{X}$. Then we have
\[
\adp(x)=p^{-d} \otimes x \mod \Zu\otimes \Fit{d-1}{X}.
\]
\end{lemma}
\begin{proof}
First, using the linearity of $\adp$ and its compatibility with push-forwards along closed embeddings, we reduce the question to the case $x=[\Oc_X]$, and $X$ an integral variety of dimension $d$.
  
Let $s \colon S \hookrightarrow X$ be the reduced closed subvariety of $X$, whose points are those $y \in X$ such that the local ring $\Oc_{X,y}$ is not a regular local ring. Let $u \colon U \to X$ be the open complementary subscheme --- the scheme $U$ is a regular variety. We have the localization sequence 
\[
\Zu \otimes \Kzl(S) \xrightarrow{s_*} \Zu \otimes \Kzl(X) \xrightarrow{u^*} \Zu \otimes \Kzl(U) \to 0.
\]
Moreover the map $\Zu \otimes \Fit{d-1}{X} \to \Zu \otimes\Fit{d-1}{U}$ is surjective, and $\adp \circ u^*=u^* \circ \adp$ by \eqref{eq:RR}. Therefore we can replace $X$ by $U$, and assume that $X$ is a regular variety.

In this situation, because of \eqref{eq:RR}, we have $\adp(x)=\Th(-\Tan_X)\cdot[\Oc_X]$. Writing $\Tan_X$ as $[E]-[F]$ for some vector bundles $E$ and $F$ over $X$, we take an open non-empty subvariety $v \colon V \to X$, where both $E$ and $F$ are free. We get 
\begin{align*}
 v^*\circ \adp(x)&=v^* \circ \Th(-\Tan_X)\\
 &=\Th \circ v^*(-\Tan_X)\\
 &=\Th \big( \rank(-\Tan_X) \cdot [\Oc_V])\\
 &=\Th(-d\cdot[\Oc_V])\\
 &=([\Oc_V]^0 + \cdots + [\Oc_V]^{1-p})^{-d}\\
 &=p^{-d}\otimes [\Oc_V]\\
 &=v^*(p^{-d}\otimes x).
\end{align*}
We conclude using the inclusion $\ker (v^*) \subset \Zu \otimes \Fit{d-1}{X}$, which follows from the localization sequence.
\end{proof}

When resolution of singularities is available, one can make a more precise statement : a consequence of Theorem~\ref{ad_main} below is that, for $x \in \Fit{d}{X}$, the element 
\[
p^{d+\left[ \frac{d}{p-1} \right]}\cdot \adp(x) - p^{\left[ \frac{d}{p-1} \right]}\otimes x
\]
belongs to the image of the map $\Fit{d-1}{X} \to \Zu \otimes \Kzl(X)$.\\

\paragraph{\bf Atiyah's decomposition}
We now state and prove the algebraic analogue of \cite[Proposition~5.6]{Ati-Po-66}.

\begin{theorem}
\label{ad_main}
Let $X$ be a variety and $x \in \Kit{d}{X}$. Assume that $X$ is generated by $\pp$-regular classes. Then we may find elements $x_k \in \Kit{d-k(p-1)}{X}$ such that we have in $\Zu \otimes \Kzl(X)$
\[
\adp(x)=\sum_{k=0}^{\left[ \frac{d}{p-1} \right]} \pp^{-d-k} \otimes x_k.
\]
Moreover, one can choose $x_0$ so that $x=x_0 \mod \Fit{d-1}{X}$.
\end{theorem}
\begin{proof}
The group $\Fit{d}{X}$ is additively generated by the subgroup $\Fit{d-1}{X}$ and elements $z$ satisfying $\lambda \cdot z=f_*[\Oc_Z]$, for some local complete intersection $d$-dimensional variety $Z$, some projective morphism $f \colon Z \to X$, and some prime to $p$ integer $\lambda$. Reasoning by induction on $d$, we see that it will be enough to prove the statement for elements $x=z$ as above. We have, using \eqref{eq:RR} and Proposition~\ref{prop:Bott}
\[
\lambda \cdot \adp(z) = \adp \circ f_*[\Oc_Z] =f_* \circ \adp[\Oc_Z]  =f_* \big( \Th(-\Tan_Z)\cdot[\Oc_Z] \big)=\sum_{k=0}^{\left[ \frac{d}{p-1} \right]} \pp^{-d-k} \otimes f_*(z_k),
\]
the element $z_k$ being a lifting of $\wg{k}{-\Tan_Z}\cdot[\Oc_Z]$ to $\Kzl(Z)$, for every $k$. Thus by \eqref{eq:gammatop}, the element $f_*(z_k)$ belongs to $\Fit{d-k(p-1)}{X}$, for every $k$.

Using Lemma~\ref{lemm:integrality}, we find an integer $n\geq d$, and $\alpha \in \Fit{d-1}{X}$ such that $\adp(z)=p^{-d} \otimes z + p^{-n} \otimes \alpha$. Pick integers $u$ and $v$ such that $\lambda u=1-\pp^{n-d} v$. Then we have in the group $\Zu \otimes \Kzl(X)$ 
\begin{align*} 
\adp(z) &= u \lambda \cdot \adp(z) + \pp^{n-d} v \cdot\adp(z) \\ 
&=p^{-d}  \big( u \otimes f_*(z_0) + v p^{n} \cdot \adp(z) \big) + \sum_{k=1}^{\left[ \frac{d}{p-1} \right]} u\pp^{-d-k} \otimes f_*(z_k). 
\end{align*}
Then we set $x_0=u \cdot f_*(z_0) + v\cdot (p^{n-d}z+\alpha)$  and $x_k=u \cdot f_*(z_k)$ for $k \geq 1$. This proves the formula.\\

Now we have 
\begin{align*}
x_0&=u \cdot f_*(z_0) + vp^{n-d}\cdot z &\mod \Fit{d-1}{X}\\
&=u\lambda \cdot z + vp^{n-d}\cdot z &\mod \Fit{d-1}{X}\\
&=z &\mod \Fit{d-1}{X}
\end{align*}
This proves the statement about the choice of $x_0$. 
\end{proof}

\section{Steenrod operations on reduced Chow groups}
In this section $X,F,\ba{F}$ will be data satisfying Condition~\ref{def:cond}. We write $\Ch=\Z/p \otimes \CH$ for the Chow group modulo $p$, and consider the \emph{reduced Chow group modulo $p$} 
\[
\Chb(X) := \im \big(\Ch(X) \to \Ch(\ba{X})\big),
\]
the image of the pull-back along the scalars extension morphism $\ba{X} \to X$. For an element $y \in \Kzl(X)$, we shall denote by $\ba{y}$ its image in $\Kzl(\ba{X})$. Similarly, we consider the group 
\[
\ba{G}(X)=\im\big(\Z/p \otimes \Grts(X) \to \Z/p \otimes \Grts(\ba{X}) \big).
\]

We have a commutative diagram
\[
\xymatrix{
\Ch_{\bullet}(X)\ar[rrr]^{\id_{\Z/p} \otimes \varphi_X} \ar[d] &&& \Z/p \otimes \Grts(X)  \ar[d]\\ 
\Ch_{\bullet}(\ba{X}) \ar[rrr]^{\id_{\Z/p} \otimes \varphi_{\ba{X}}} &&& \Z/p \otimes \Grts(\ba{X}).
}
\]
There is an induced morphism $\Chb(X) \to \ba{G}(X)$, which is both injective, as a restriction of the isomorphism $\id_{\Z/p} \otimes \varphi_{\ba{X}}$ (Remark~\ref{rem:torfree}), and surjective, because  $\id_{\Z/p} \otimes \varphi_{X}$ is so (Proposition~\ref{prop:phi}, \eqref{phiisom}).\\

We now construct operations $\ba{G}(X) \to \ba{G}(X)$. Let $d$ be an integer and $x \in \Fit{d}{X}$. Choose elements $x_k \in \Fit{d-k(p-1)}{X}$, for $k=0,\cdots,[d/(p-1)]$, as in Theorem~\ref{ad_main}. Their images in $\Z/p \otimes \Grts(\ba{X})$ do not depend on the choice of the elements $x_k$, by Proposition~\ref{cor:injl} applied with $u_k=x_k$ and $u_{-1}=0$. 

If $\ba{x} \in \Fit{d-1}{\ba{X}}$, then consider elements $x_k' \in \Fit{d-1-k(p-1)}{\ba{X}}$ associated with $\ba{x}$ by Theorem~\ref{ad_main}, for the variety $\ba{X}$. It follows from Proposition~\ref{cor:injl} (applied with $u_{-1}=0$) that each element $u_k=p\cdot x_k'-\ba{x_k}$ has image zero in $\Z/p \otimes \Grt{d-k(p-1)}{\ba{X}}$. Since $p\cdot x_k'$ is zero in this group, we see that the element $x_k$ has also image zero in $\Z/p \otimes \Grt{d-k(p-1)}{\ba{X}}$.

Similarly, if $\ba{x}$ is divisible by $p$ in $\Fit{d}{\ba{X}}$, we see that each element $x_k$ has also image zero in $\Z/p \otimes \Grt{d-k(p-1)}{\ba{X}}$.\\

We have proven that the association $x \mapsto \ba{x_k}$ induces an endomorphism of the group $\ba{G}(X)$ which lowers the degree by $k(p-1)$. Using the identification with reduced Chow groups modulo $p$ obtained above, we get \emph{reduced Steenrod operations}
\[
\Sq^X_k \colon \Chb_{\bullet}(X) \to \Chb_{\bullet-k(p-1)}(X).
\]
We shall denote by $\Sq^X=\sum_k \Sq^X_k$ the total operation. We prove now some of the expected properties of this operation. 

\begin{proposition}
\label{prop:st}
Let $X$ and $Y$ be varieties over a field $F$, satisfying Condition~\ref{def:cond}. 
\begin{enumerate}[(a)]
\item \label{st:proper} Let $f \colon X \to Y$ be a proper morphism. Then $f_* \circ \Sq^X=\Sq^Y \circ f_*$.
\item \label{st:lci} Let $f \colon X \to Y$ be a local complete intersection morphism, with virtual tangent bundle $\Tan_f \in \Kzh(X)$. Then $\Sq^X \circ f^*=\wchs(-\Tan_f) \circ f^* \circ \Sq^Y$.
\item \label{st:extprod} Assume that $X \times Y$ satisfies Condition~\ref{def:cond}. We have $\Sq^{X \times Y}=\Sq^X \times \Sq^Y$.
\end{enumerate}
\end{proposition}
\begin{proof}
These statements are consequences of the corresponding statements for the homological Adams operation and Proposition~\ref{prop:phi}; in addition Proposition~\ref{prop:Bott} and \eqref{eq:RR} are used to prove \eqref{st:lci}.
\end{proof}

\begin{proposition}
\label{prop:stzero}
Let $X$ be a variety satisfying Condition~\ref{def:cond}. Then 
\[
\Sq^X_0=\id_{\Chb(X)}.
\]
\end{proposition}
\begin{proof}
This is a consequence of the last part of Theorem~\ref{ad_main}.
\end{proof}
When the data $X,F,\ba{F}$ satisfy Condition~\ref{def:cond}, and additionally $X$ is a local complete intersection variety, with virtual tangent bundle $\Tan_X \in \Kzh(X)$, we define the \emph{total cohomological reduced Steenrod operation} by the formula $\Sq_X=\wchs(\Tan_X) \circ \Sq^X$. Its $k$-th individual component $\Sq_X^k$ lowers dimension by $k(p-1)$.

\begin{proposition}
\label{prop:stc}
Let $X$ and $Y$ be local complete intersection varieties over a field $F$, both satisfying Condition~\ref{def:cond}. Then
\begin{enumerate}[(a)]
\setcounter{enumi}{4}
\item \label{st:ring} If $X$ is smooth over $F$, then the map $\Sq_X$ is a ring homomorphism.
\item \label{st:cohlci} Let $f \colon X \to Y$ be a local complete intersection morphism. Then $\Sq_X \circ f^*=f^* \circ \Sq_Y$.
\item \label{st:rr} Let $f \colon X \to Y$ be a proper morphism, and set $\Tan_f=\Tan_X-f^*\Tan_Y\in\Kzh(X)$. Then $\Sq_Y \circ f_*=f_* \circ \wchs(-\Tan_f) \circ \Sq_X$.
\item \label{st:cohextprod} Assume that $X \times Y$ satisfies Condition~\ref{def:cond}. Then $\Sq_{X \times Y}=\Sq_X \times \Sq_Y$. 
\item \label{st:zero} We have $\Sq_X^0=\id_{\Chb(X)}$.
\end{enumerate}
\end{proposition}
\begin{proof}
Statement~\eqref{st:cohlci} follows from Proposition~\ref{prop:st}, \eqref{st:lci}, Statement~\eqref{st:cohextprod} follows from Proposition~\ref{prop:st}, \eqref{st:extprod}, and Statement~\eqref{st:rr} follows from Proposition~\ref{prop:st}, \eqref{st:proper}.

Concerning \eqref{st:ring}, compatibility with products follows from the combination of \eqref{st:cohlci} and ~\eqref{st:cohextprod}, and compatibility with unities follows from  Proposition~\ref{prop:stzero} applied with $X=\pt$, and \eqref{st:cohlci}. 

Finally \eqref{st:zero} follows from the fact that $\wch{0}{\Tan_X}=\id$, and Proposition~\ref{prop:stzero}. 
\end{proof}

\begin{proposition}
Let $X$ be a smooth connected variety satisfying Condition~\ref{def:cond}, and $x \in \Chb^q(X)$. Then $\Sq_X^q(x)=x^p$, and $\Sq_X^k(x)=0$ for $k > q$.
\end{proposition}
\begin{proof}
Choose $y \in \Fit{\dim X -q}{X}$ representing $x$. Since $X$ is a regular variety, there an element $y' \in \Kzh(X)$ satisfying $y'\cdot[\Oc_X]=y$, and we have 
\[
1 \otimes \big(\ad(y') \cdot [\Oc_X]\big)=\Th(\Tan_X) \cdot \adp(y).
\]
  
Using Proposition~\ref{prop:Bott} and Theorem~\ref{ad_main}, we find elements $t_i \in \Fig{i(p-1)}{X}$ and $y_j \in \Fit{\dim X -q-j(p-1)}{X}$ such that
\begin{align*} 
1 \otimes \big( \ad(y') \cdot [\Oc_X]\big)&=\Big( \sum_{i=0}^{\left[ \frac{\dim X}{p-1} \right]} p^{\dim X-i} \otimes t_i \Big)\cdot \Big( \sum_{j=0}^{\left[ \frac{\dim X-q}{p-1} \right]} p^{q-\dim X -j} \otimes y_j \Big)\\
&=\sum_{k=0}^{\left[ \frac{\dim X}{p-1} \right]} p^{q-k} \otimes z_k \quad \text{with} \quad z_k= \sum_{i+j=k} t_i \cdot y_j. 
\end{align*}
Note that $z_k \in \Fit{\dim X-q-k(p-1)}{X}$ by \eqref{eq:gammatop}. It follows from Lemma~\ref{lemm:psipower} below and Proposition~\ref{cor:injl} that the element $z_q-y^p$ (resp. $z_k$ for $k >q$) has image zero in $\Z/p \otimes \Grt{\dim X -qp}{\ba{X}}$ (resp. $\Z/p \otimes \Grt{\dim X -q-k(p-1)}{\ba{X}}$.

Now the element $y_j$ (resp. $t_i\cdot[\Oc_X]$) has image $\Sq^X_j(x)$ (resp. $\wch{i}{\Tan_X}$) in $\Chb(X)$, hence $z_k$ has image $\Sq_X^k(x)$. The statement follows. 
\end{proof}

\begin{lemma}
\label{lemm:psipower}
Let $X$ be a variety, and $x \in \Kzh(X)$. Then $\ad(x)=x^p \mod p\Kzh(X)$
\end{lemma}
\begin{proof}
The maps $\ad$ and $x \mapsto x^p$ are linear modulo $p$, compatible with pull-backs, and coincide on the classes of line bundles by the very definition of the Adams operation. We conclude using the splitting principle.
\end{proof}

\begin{remark}
We need to resolve singularities on the varieties corresponding to the prime cycles to which we want to apply the Steenrod operations. When the base field is differentially finite over a perfect subfield, but does not necessarily admit $p$-resolution of singularities, one can still resolve singularities of low dimensional varieties by the results of  \cite{Co-Pi-I, Co-Pi-II}. It follows that, when the integer $p$ is small enough ($p=2$ or $3$) and $X$ is a projective homogeneous variety of arbitrary dimension over a field as above, the method presented here can be used to construct operations on low dimensional cycles modulo $\pp$ (namely the operations $\Chb_i(X) \to \Chb_{i-k(p-1)}(X)$ for $i\leq 3$ and $k(p-1)\leq i$).
\end{remark}

\section{A degree formula}
\begin{theorem}
\label{th:degree}
Let $X$ be a projective variety, generated by $p$-regular classes. Let $x \in \Fit{d}{X}$. Then $X$ possesses a zero-cycle of degree
\[
\lambda \cdot p^{\left[ \frac{d}{p-1}\right]} \cdot \deg(x),
\]
for some integer $\lambda$ prime to $p$. (See \eqref{eq:deg} for a definition of the map $\deg$.)
\end{theorem}
\begin{proof}
By induction on $d$, the case $d=0$ following from \eqref{eq:degcomp}.
  
Assume that the statement holds when $d<e$, and take $x \in \Fit{e}{X}$. By Theorem~\ref{ad_main}, we find elements $x_k \in \Fit{e-k(p-1)}{X}$ satisfying
\[
\adp(x)=\sum_{k=0}^{\left[ \frac{e}{p-1}\right]} p^{-e-k} \otimes x_k.
\]
Then applying the degree map $\deg \colon \Zu \otimes \Kzl(X) \to \Zu$,
\[
(p^e-1) \deg(x)=p^e \deg \circ \adp(x) -\deg(x)=\deg(x_0-x) + \sum_{k=1}^{\left[ \frac{e}{p-1}\right]} p^{-k} \deg(x_k).
\]
Since $x_0-x \in \Fit{e-1}{X}$, we know by induction hypothesis that there is a cycle $c_0 \in \Z_{(p)} \otimes \CH_0(X)$ whose degree is $p^{\left[ \frac{e}{p-1}\right]}\deg(x_0-x)$. Similarly there is, for every $k\geq 1$, a cycle $c_k \in \Z_{(p)} \otimes \CH_0(X)$  whose degree is  $p^{\left[ \frac{e}{p-1}\right]-k}\deg(x_k)$. Then $(p^e-1)^{-1} \cdot ( c_0 +\cdots + c_{\left[ \frac{e}{p-1}\right]})$ is an element of $\Z_{(p)} \otimes \CH_0(X)$ whose degree is $p^{\left[ \frac{e}{p-1}\right]}\deg(x)$. This concludes the proof.
\end{proof}

When $X$ is a projective variety over a field $F$, we denote by 
\[
\chi(\Oc_X)=\deg[\Oc_X]=\sum_{i\geq 0} (-1)^i \dim_F H^i(X,\Oc_X)
\]
the Euler characteristic of the structure sheaf of $X$. 

The following statement is the $p$-primary version of \cite[Theorem~4.4]{Zai-09} (but we also prove it for singular varieties). 

\begin{proposition}
Let $f \colon X \to Y$ be a projective morphism of integral varieties of dimension $d$. Define an integer $\deg f$ as the degree of the map $f$ if it is generically finite, or zero otherwise. Assume that $Y$ is projective and generated by $p$-regular classes. Then $Y$ possesses a zero-cycle of degree
\[
\lambda \cdot p^{\left[\frac{d-1}{p-1}\right]}\big(\chi(\Oc_X) - (\deg f)\cdot\chi(\Oc_Y)\big),
\]
for some integer $\lambda$ prime to $p$.
\end{proposition}
\begin{proof}
There is an element $\delta \in \Fit{d-1}{X}$ such that we have in $\Kzl(Y)$ the equation $f_*[\Oc_X]=(\deg f) \cdot [\Oc_Y] + \delta$. Then
\[
\chi(\Oc_X)=\deg [\Oc_X]=\deg \circ f_* [\Oc_X]=(\deg f) \cdot \chi(\Oc_Y) + \deg(\delta).
\]
Using Theorem~\ref{th:degree}, we see that the integer $p^{\left[\frac{d-1}{p-1}\right]} \cdot \deg(\delta)$ is the degree of an element of $\Z_{(p)} \otimes \CH_0(Y)$, as requested. 
\end{proof}

\bibliographystyle{alpha}

\end{document}